\documentclass[a4paper]{article}

\usepackage{amssymb,amsmath}
\usepackage{latexsym}
  \usepackage[backend=bibtex, style=numeric]{biblatex}
\bibliography{poly}

\title{\bf \boldmath }

\newcommand{\ds}{\displaystyle}

\newtheorem{theorem}{Theorem}[section]
\newtheorem{defin}[theorem]{Definition}
\newtheorem{lemma}[theorem]{Lemma}
\newtheorem{example}[theorem]{Example}
\newtheorem{prop}[theorem]{Proposition}
\newtheorem{cor}[theorem]{Corollary}

\newcommand{\Aut}{{\rm Aut}}
\newcommand{\vh}{\vskip 0.5cm \noindent }
\newcommand{\vq}{\vskip 0.25cm \noindent }

\newcommand{\Hom}{{\rm Hom}}

\newcommand{\Tor}{{\rm Tor}}
\newcommand{\N}{{\mathbb{N}}}
\newcommand{\Z}{{\mathbb{Z}}}

\newcommand{\bi}{\begin{itemize}}
\newcommand{\ei}{\end{itemize}}

\newcommand{\mapright}[1]{{\smash{\mathop{\longrightarrow}\limits^{#1}}}}

\newcommand{\rmapdown}[1]{{\Big\downarrow\rlap{$\vcenter{\hbox{$\scriptstyle#1$}}$}}}
\newcommand{\lmapdown}[1]{{\vcenter{\hbox{$\scriptstyle#1$}}\Big\downarrow
               \vcenter{\hbox{$\phantom{\scriptstyle#1}$}}}}

\newcommand{\ba}{\begin{array}}
\newcommand{\ea}{\end{array}}

\newcommand{\comdiag}[1]{\begin{matrix}#1\end{matrix}}

\newenvironment{proof}{\noindent{\bf Proof.}}{\quad\medskip}

\begin{document}

\title{ The nilpotent genus   of  finitely generated residually nilpotent  groups.}
\author{ Niamh O'Sullivan}
\date{}
\maketitle
\begin{abstract} \noindent  If $G$ and $H$   are finitely generated residually nilpotent   groups, then $G$ and $H$ are in the same nilpotent genus if they have the same lower central
   quotients (up to isomorphism).  A stronger condition is that $H$ is para-$G$ if there exists a monomorphism of $G$ into $H$ which induces isomorphisms between the corresponding quotients of their lower central series. We first consider residually nilpotent groups and find sufficient conditions on the monomorphism so that $H$ is para-$G.$  We then  prove that for certain polycyclic groups, if $H$ is para-$G$, then  $G$ and $H$ have the same Hirsch length.  We  also prove  that the pro-nilpotent completions  of these polycyclic groups are   locally polycyclic. 
\end{abstract}

\section{Introduction} 

\par Throughout this paper the groups we study will be  finitely generated residually nilpotent   groups. If $G$ and $H$   are two such groups, then $G$ and $H$ are in the same nilpotent genus if they have the same lower central
   quotients (up to isomorphism). Hanna Neumann asked whether free groups can be characterised in terms of their lower central series and as a result Baumslag \cite{baum}, Bridson and Reid \cite{br}, among others, looked at parafree groups: residually nilpotent groups in the same nilpotent genus as a free group.   If $H$ is a residually finite group and  $\phi : G\rightarrow  H$ is the inclusion of a subgroup 
	$G$, then
$(H, G)_{\phi}$ is called a Grothendieck Pair if the induced homomorphism $\hat \phi : \hat G \rightarrow \hat H$ is an isomorphism 
but $\phi$ is not, where $\hat H$ denotes the profinite completion of $H$. Bridson and Reid showed, Prop. 3.4. \cite{br},  that if $(G, H)_{\phi}$ is a Grothendieck
pair of residually nilpotent groups, then $\phi$ induces isomorphisms on the lower central quotients.    In \cite{bmo1},\cite{bmo2},  Baumslag, Mikhailov and Orr studied the nilpotent genus of finitely generated residually nilpotent metabelian groups and the motivation for this work was to extend some of their results. 
 \begin{defin} The nilpotent genus of $G$  is the set of all isomorphism classes of groups $H$ such that $G/\gamma_i (G)\cong H/\gamma_i (H),$ for all $i\in \N.$
\vq  The group $H$ is para-$G$ if there exists a monomorphism $\phi :G\rightarrow H$ which induces isomorphisms between the corresponding quotients of their lower central series: $$\phi_i :G/\gamma_i (G) \mapright{\cong} H/\gamma_i (H),$$  we say that $H$ is para-$G$ via $\phi.$ \end{defin}	 Given any set of primes $\pi$, let $\pi'$ denote its complement in the set of all primes. If $n\in \N,$ denote by  $\pi(n)$   the set of primes dividing $n$. We say that $n$ is a $\pi$-number if $\pi(n)\subseteq \pi$ and  a $\pi'$-number if $\pi(n)\subseteq \pi'.$ If   $K\le G$, then the $\pi$-isolator of $K$ in $G$ denoted $I_{\pi} (K)$ is $$I_{\pi} (K)= < x\in G\mid x^n\in  K, \,  n\;{\rm is\; a}\;\pi-{\rm no}\,>.$$ If $K\triangleleft G$ and $G/K$ is a finitely generated nilpotent group, then $$ I_{\pi} (K)=\{  x\in G\mid x^n\in  K, \,  n\;{\rm is\; a}\;\pi-{\rm no}\,\} \triangleleft G$$ and $I_{\pi} (K)/K$ is a finite $\pi-$group.  Let $$\tau (G)= \ds \bigcup_{i=1}^\infty \pi (|\Tor(\gamma_i (G)/\gamma_{i+1} (G))|).$$	
\begin{defin} Let $\tau=\tau (G).$  Suppose that  $\phi:G\rightarrow H$ be  a monomorphism that  induces an isomorphism $G_{ab}\cong H_{ab}$ and    there exists   $\tau'$-numbers $n_{i}$ such that   $\gamma_i (H)^{n_i}\le \phi(\gamma_i (G))$, for all $i\ge 2 .$  Then we call 
$\phi$ a lower central  $\tau$-monomorphism.  \end{defin}
 We shall prove that  the existence of such a monomorphism is sufficient for $H$ to be para-$G.$
 \vh {\bf Theorem 2.2.} {\it   Let $G,\; H $ be finitely generated residually nilpotent groups  and set  $\tau=\tau (G)$. Suppose that     $\phi: G\rightarrow  H$  is  a lower central  $\tau$-monomorphism. Then $H$ is para-$G$ via $\phi$ and $\tau (H)=\tau$. }

  In   \cite{bmo2},  Baumslag, Mikhailov and Orr  found examples of groups $G$ such that the nilpotent genus of $G$ is nontrivial.  We show that  the nilpotent genus of   $G$  is nontrivial if the nilpotent genus of a certain type of quotient is nontrivial. 
  \vh {\bf Proposition 2.6.} {\it  Let $G,\; H $ be finitely generated residually nilpotent groups  where  $\tau=\tau (G)$ is finite.  Suppose that     $\phi: G\rightarrow  H$  is a  lower central  $\tau$-monomorphism. Let    $N=I_{\pi} ( \gamma_k (H))$   and $M=I_{\pi}(\gamma_k (G))$   where $\pi\subseteq \tau$.  If $\tau (G/Z_j(M)) \subseteq \tau ,$ for some $j\ge 1   ,$ then    $H/Z_j(N)$ is para $G/Z_j(M).$  Furthermore  if     $G/Z_j(M) \not\cong H/Z_j(N)  ,$ then $ G\not  \cong H.$}
  
  \vh  In \S 3 and \S 4   we will study  the class ${\cal P}^*$ ( ${\cal P}_c^*$) of polycyclic groups which are residually nilpotent and nilpotent (of class $c$)  by abelian. Baumslag, Mikhailov and Orr \cite{bmo2} proved that if  $G,H  $ are finitely generated residually nilpotent metabelian groups, then $H$ is para-$G$, if, and only if, $G$ is para-$H$ so that this is an equivalence relation.   Unfortunately we do not have as strong a result  when $G$ and $H$ are polycyclic of derived length greater than $2$, however we can use their work to show that, for certain polycyclic groups, if $H$ is para-$G$, then $h(G)=h(H),$  which is  a  necessary condition for $G$ to be also para-$H$. 
 \vh {\bf Theorem 3.4.} {\it    Let $G,H\in {\cal P}_c^*$ where $H$ is para-$G$ via $\phi$ and one of the following holds: \bi \item[(i)] $    G/G'',\;  H/H''$ are residually nilpotent;
\item[(ii)]  $|Z_{c-1} (G'): G''|,\; | Z_{c-1} (H'): H''|<\infty .$ \ei Then $h(H)=h(G)$.} 
\vh  
      In \cite{bmo2}  Baumslag, Mikhailov and Orr proved that the pro-nilpotent completion of a  ${\cal P}_1^*$  group is locally polycyclic, in \S 4  we extend their result to   ${\cal P}^*.$	
      \vq {\bf Theorem 4.4.} {\it  If     $G\in {\cal P}^*$, then $\hat G_{nil}$ is locally polycyclic.}

\section{The nilpotent genus of residually nilpotent groups}

	When considering the nilpotent genus, the lower central nilpotent quotients of the groups play a significant role.  In later  sections,  we'll also  consider polycyclic groups where the derived subgroup is a finitely generated nilpotent group.  If $G$ is any group with operator domain $\Omega$, then  $\gamma_{i} (G)/\gamma_{i+1} (G)$ is a right $\Omega$-module  and the mapping $$a\gamma_{i+1}(G)\otimes_{\Z} gG'\mapsto [a,g]\gamma_{i+2} (G)$$ is a well-defined epimorphism from $\gamma_i(G)/\gamma_{i+1} (G)\otimes_{\Z} G_{ab} $ to $\gamma_{i+1}(G)/\gamma_{i+2}(G)$ so that $G_{ab}$ greatly influences the structure  of $G/\gamma_{i}(G)$, for all $i$,  (see Robinson p. 131 \cite{rob}). Some of the results that we use (which can be easily proved using induction and most likely  appear elsewhere) are included in   the following Lemma:
	\begin{lemma}	Let $G$ be a  group and let $N\le G.$   \bi \item[(i)]  If      $NG'=G,$ then   the inclusion $N \hookrightarrow G$  induces  epimorphisms $$N/\gamma_i (N)\twoheadrightarrow G/\gamma_i (G);$$
		\item[(ii)]  If  $G$ is nilpotent and     $G^m \le NG',$  for some $m,$ then  $\gamma_i( G)^{n_i} \le \gamma_i (N)$ and $G^n\le N$, for some $\pi(m)$-numbers $n_i,\,n;$   
		\item[(iii)]  If  $G$ is nilpotent of class $c$ and $\alpha\le \Aut (G)$ is abelian where the automorphism induced by  $\alpha $   on $G_{ab}$ is the identity,  then the automorphism induced by  $\alpha $   on $ \gamma_i(G)/\gamma_{i+1} (G)$  is the identity  and $G\rtimes <t>$   is nilpotent of class $c,$  where $t$ acts on $G$ by $\alpha$; 
	\item[(iv)]  If $G$ is a polycyclic group, then there exists $k$ such that $|\gamma_k (G)/\gamma_{k+j} (G)|<\infty ,$ for all $j \ge 1,$  $\pi(|\gamma_k (G)/\gamma_{k+j} (G)|)\subseteq \pi( |\gamma_k (G)/\gamma_{k+1} (G)|)$ and $$\tau (G)=\bigcup_{i=1}^k \pi (|\Tor(\gamma_i (G)/\gamma_{i+1} (G))|) $$  is a finite set of primes. 	\ei 
	\end{lemma} \begin{proof} (i) See Corollary 10.3.3. p. 155 \cite{hall}.
		\vq (ii) This is true if $G$ is abelian. Let $G$ be nilpotent of class $k$ so that $\gamma_k(G)\le Z(G)$ and $[x^l,y]=[x,y^l]=[x,y]^l, $ for all $x\in \gamma_{k-1}(G),\,y\in G,\,l\in \Z.$   Assume that the result  is true for  nilpotent  groups of class $k-1$. Then there exists $\pi (m)-$numbers $  m_{i },\,l$ such that  $$\gamma_{i } (G)^{m_{i }}\le \gamma_{i } (N)\gamma_k (G),\;\;G^l\le N \gamma_{k}(G),\;\; \forall \;1\le i\le k-1.$$ Therefore  
		  $$\ba{l} \ds \gamma_k(G)^{m_{k-1}l}\le  [\gamma_{k-1} (G)^{m_{k-1}}, G^l] \le [\gamma_{k-1}(N)\gamma_k (G) , N]\le  \gamma_{k}(N) \\ \\ \Longrightarrow  \ds   \gamma_i( G)^{n_i} \le \gamma_i (N),\;\;G^n\le N\ea $$ where  $n_k=m_{k-1}l$,   $n_i=m_i n_k,$ for all $i<k$, and $n=n_{k}l.$
\vq (iii) It is very straightforward to show this by induction.
\vq (iv) Since $G$ is a polycyclic group, there exists $k$ such that $h(\gamma_k (G))=h(\gamma_j (G)),$ for all $j\ge k.$ Let $n=|\gamma_k(G):\gamma_{k+1} (G)|.$ Then  	$\gamma_{k+j-1} (G)/\gamma_{k+j} (G)  $  has exponent dividing $n$, for all $j\ge 1$, and $\tau(G)$ is a finite set of primes.	$\square$ 	\end{proof}

\begin{theorem} Let $G,\; H $ be finitely generated residually nilpotent groups  and set  $\tau=\tau (G)$. Suppose that     $\phi: G\rightarrow  H$  is  a lower central  $\tau$-monomorphism. Then $H$ is para-$G$ via $\phi$ and $\tau (H)=\tau$.       
\end{theorem}
\begin{proof} By the hypothesis,  $\phi$ induces an isomorphism $G_{ab}\cong H_{ab}$  and therefore  induces  epimorphisms $\phi_i: G/\gamma_{i} (G)\rightarrow H/\gamma_i (H),$ by Lemma 2.1 (i).   Suppose that $g\in G$ and $ \phi(g)\in \gamma_i (H),$  then  $ g^{n_i}\in \gamma_{i} (G) ,$ where  $n_{i}$ is a   $\tau'$-number,   thus   $ g\in \gamma_i (G)$  and  $\phi_i$ is an isomorphism for all $i.$ Therefore $H$ is para-$G$ via $\phi$ and $\tau (H)=\tau $. 
$\square$ \end{proof}
\vq The   conditions in Theorem 2.2 can be simplified  when $H'$ is nilpotent of class $\le 2.$
\begin{cor}   Let $G,\; H $ be finitely generated residually nilpotent groups  where $H'$ is nilpotent of class at most $2$    and set  $\tau=\tau (G)$. Suppose that   a  monomorphism  $\phi: G\rightarrow  H$ induces an isomorphism $G_{ab}\cong H_{ab}$ and   $\gamma_2 (H)^{n }\le \phi(\gamma_2 (G))$, for some $n\in \tau' .$   Then $H$ is para-$G$ via $\phi.$  
\end{cor}
\begin{proof}  We will   replace $G$ by $\phi(G)$ so that we are assuming that $G\le H$. If $h\in H,$ then $h=gk,$ for some  $g\in G,\, k\in H'$,  since the inclusion induces an isomorphism $G_{ab} \cong H_{ab}$.
\vq  Suppose that  $H'$ is abelian  and  $\gamma_i (H)^{n}\le \gamma_i (G)$, for some  $i\ge 2$.  If $x\in \gamma_i(H),$  $h\in H,$ then $[x^l,h]  = [x,h]^l$, for all $l.$ 
Therefore 
$$\gamma_{i+1} (H)^n \le [\gamma_i(H)^n , H]\le [\gamma_i(G) ,G H']= [\gamma_i(G) ,G ]=\gamma_{i+1} (G) .$$ 

 \vskip 0.25cm  Suppose that $H'$ is nilpotent of class $2$ and   $\gamma_i (H)^{m}\le \gamma_i (G)$, for some  $i\ge 2$ and some $\pi(n)$-number  $m $ where  $n|m.$ Set $$a=\left\{ \begin{matrix} 1, &  {\rm if} \; 2\not| n \cr   2 , &  {\rm if} \; 2 \;| n \end{matrix} \right. .$$  Since $H'$ is nilpotent of class $2,$    
$$[x,y]^l=[x^l,y]=[x,y^l], \;\;(xy)^{l}=x^ly^{l}[y,x]^{l(l-1)/2}, $$ for all   $ l\in \N ,  \;x,y\in H'.$ This implies that if $x,y\in \gamma_{i+1} (H)$ and $x^l,\,y^l\in \gamma_{i+1} (G),$ then $(xy)^{al^2} \in  \gamma_{i+1} (G) $ for any $\pi (n)-$number $l$, where $n|l.$    Given any $r\in N,$   $$\ba{ll} \gamma_{i+1}(H)^{rm^2} &\ds =[\gamma_i(H), H]^{rm^2}=[\gamma_i(H), H'G]^{rm^2}= [\gamma_i(H),G]^{rm^2}[\gamma_i(H), H']^{rm^2}\\ &\\ &\ds = [\gamma_i(H),G]^{rm^2}[\gamma_i(H)^m, (H')^m]^r\le  [\gamma_i(H),G]^{rm^2} \gamma_{i+1} (G).\ea $$   If  $x\in \gamma_i( H), g \in G, \;l\in\N,$ then      $$\ba{ll} & [x^l, g]=[x,g]^{x^{l-1}+\dots+x+1}= [x,g]^l \prod_{j=1}^{l-1} [x,g,x^j]= [x,g]^l[x,g,x^{l(l-1)/2} ]= [x,g]^l[[x,g]^{l(l-1)/2},x ]\\ & \\ \Longrightarrow & [x,g]^{am^2}=[x^{am^2}, g][[x,g]^m, x^{-am(m^2-1)/2}]\in \gamma_{i+1} (G) .\ea$$ Set $ n_{i+1}={a^3m^4},$ then     $$ [\gamma_i(H),G]^{n_{i+1}}\le \gamma_{i+1} (G) \Longrightarrow   \gamma_{i+1}(H)^{n_{i+1}}\le \gamma_{i+1} (G) .  $$      Note that $n|n_{i+1}$ and $n_{i+1}$ is a $\pi(n)-$number.

\vskip 0.5cm   Therefore when $H'$ is nilpotent of class $\le 2$ and $n$ is a $\tau'$-number the hypothesis of Theorem 2.2  is satisfied and thus $H$ is para-$G$ via the inclusion. $\square$

\begin{lemma} Suppose that     $G  $ is a residually-nilpotent   group  and   $  M \triangleleft G$ where $G/M$ is nilpotent. If  
$$R_j=G/Z_j (M),\;\;\;M_j=  M/Z_j (M),$$  then $R_j$ is residually nilpotent, $R_j/M_j$ is nilpotent  and $R_{j+1} \cong  R_j /Z(M_j).$
\end{lemma} \begin{proof} Given any $j\ge 1,$  $M_j \triangleleft R_j  ,$  $$R_j/M_j=(G/Z_j(M))/(M/Z_j(M))\cong G/M$$ is nilpotent and $$\ba{ll} & \ds Z(M_j)=Z(M/Z_j (M))=Z_{j+1}(M)/Z_j (M) \\ &\\ \Longrightarrow  &\ds R_{j+1} =G/Z_{j+1} (M) \cong (G/Z_j(M))/(Z_{j+1}(M)/Z_j(M))=  R_j /Z(M_j) .\ea$$
\par  Let $\rho:G\rightarrow R_1$ denote the projection.  If $\rho(x) \ds \in\bigcap_{i=1}^\infty   \gamma_i (R_1) ,$ then $$\ds x\in\bigcap_{i=1}^\infty   Z(M)\gamma_i (G)\le M\Longrightarrow [x,y] \in \ds \bigcap_{i=1}^\infty    \gamma_i (G)=1,\;\;{\rm for \;all} \;y\in M,$$  hence $x\in Z(M)$ and  $R_1$ is residually nilpotent.
\vq  Suppose that  $R_j $ is residually nilpotent, for some $j\ge 1,$ then, by the above,  $R_{j+1}=R_j/Z(M_j)$ is also residually nilpotent.  $\square $
\end{proof} 	

\begin{lemma} Suppose that  $G, \,H $ are residually nilpotent  groups and $H$ is para-$G$ via $\phi.$  If   $ N \triangleleft H$ where $H/N$ is nilpotent and $M=\phi^{-1} (N)$, then $\phi^{-1} (Z_j (N)))=Z_j (M)$ and $\phi$ induces a monomorphism $\mu_j :G/Z_j(M)\rightarrow H/Z_j (N),$   for all $1\le j .$ 
		\end{lemma}
		\begin{proof}   Suppose that $\gamma_k (H)\le N,$ then $\gamma_k (G)\le M$.       Trivially $\phi^{-1} (Z_0 (N))=Z_0(M).$ Assume that $\phi^{-1}  (Z_j (N))= Z_j (M),$ for some $j\ge 0.$  If  $\phi (x)\in  Z_{j+1} (N),$ then $x\in Z_{j+1} (M)$ and $\phi^{-1} (Z_{j+1} (N))\le Z_{j+1} (M).$     If   $x\in Z_{j+1}  (M) , y\in N,$  then there exists $x_i\in M$ such that $\phi(x_i)\gamma_i (H)=y \gamma_i (H)  ,$ for all $i\ge k ,$ therefore $[x,x_i]\in Z_j(M)$, $\phi([x,x_i])\in Z_j (N)$   and    $$[\phi(x),y]\gamma_{i+1} (H)=[\phi(x),\phi(x_i)]\gamma_{i+1} (H)\Longrightarrow [\phi (x),y]  \in \bigcap_{i=k}^\infty \gamma_{i+1} (H)Z_j (N)=Z_j (N),$$ by Lemma 2.4. Therefore $\phi (x)\in Z_{j+1}(N)$ and  $  \phi^{-1} (Z_{j+1} (N))=Z_{j+1} (M).$  
	
\vq Since $\phi^{-1} (Z_j (N)))=Z_j (M),$  $\phi$ induces a monomorphism $\mu_j :G/Z_j(M)\rightarrow H/Z_j (N),$ for all $1\le j .$ 		$\square $
\end{proof}
\begin{prop} Let $G,\; H $ be finitely generated residually nilpotent groups  where  $\tau=\tau (G)$ is finite.  Suppose that     $\phi: G\rightarrow  H$ is  a   lower central  $\tau$-monomorphism. Let    $N=I_{\pi} ( \gamma_k (H))$   and $M=I_{\pi}(\gamma_k (G))$   where $\pi\subseteq \tau$.  If $\tau (G/Z_j(M)) \subseteq \tau ,$ for some $j\ge 1   ,$ then    $H/Z_j(N)$ is para $G/Z_j(M).$  Furthermore  if     $G/Z_j(M) \not\cong H/Z_j(N)  ,$ then $ G\not  \cong H.$
\end{prop}
\begin{proof}  The monomorphism $\phi:G\rightarrow H$ satisfies the hypothesis of Theorem 2.2 and therefore $H$ is para-$G$ via $\phi.$ Hence $\phi$ induces an isomorphism $G/\gamma_k (G)\cong H/ \gamma_k (H)$ which implies that   $\phi^{-1} (\gamma_k (H))=\gamma_k (G),$ $\phi^{-1}(N)=M$ and $|N/\gamma_k (H)|=|M/\gamma_k (G)|.$ Set $m=|N/\gamma_k (H)|,$ 
$$R=G/Z_j(M),\; M_j=M/Z_j (M),\;S=H/Z_j(N),\; N_j= N/Z_j(N).$$   By Lemma 2.4, $R,\, S  $ are residually nilpotent.   By Lemma 2.5, $  Z_j (M)=\phi^{-1}( Z_j (N))$  and     $\phi :G\rightarrow  H$ induces a monomorphism $\mu : R \rightarrow S.$   
\vq  Let $\rho :G\rightarrow R$ denote the projection. If $\rho (x)\in I_\pi (\gamma_k (R))$, then, for some $\pi$-number $a$,  $$x^a\in Z_j(M)\gamma_k (G)\le M \Longrightarrow x\in M$$ and $I_\pi (\gamma_k (R))=M_j.$  Similarly $I_\pi(\gamma_k (S))=N_j.$
  
\vq    By the hypothesis $(m, n_k)=1$ and we can find $\alpha ,\,\beta \in \Z$ such that  $n=\alpha n_k=1 -\beta m.$  Suppose that $\phi (g)\in Z_j(N)\gamma_k (N)$,  then $g\in M,\;\phi(g^n)  Z_j (N) =\phi(u)Z_j(N),$ for some $u\in \gamma_k (G),$ and therefore $$\phi(g^nu^{-1})\in Z_j(N)\Longrightarrow g^nu^{-1}\in Z_j(M) \Longrightarrow  g=g^nu^{-1}u g^{\beta m} \in Z_j(M)\gamma_k (G). $$  Hence $\phi^{-1} (Z_j(N)\gamma_k (H))=Z_j(M)\gamma_k (G)$ and $\mu$ induces   an isomorphism    $$ R /\gamma_k (R )\cong  G/Z_j(M)\gamma_k (G)  \cong H/Z_j(N)\gamma_k (H) \cong S /\gamma_k (S )  $$  and hence an isomorphism $R_{ab}\cong S_{ab}.$ Given any $i\ge 2$,  we have the following commutative diagram $$\comdiag {  \gamma_i (G) &\mapright{}& \gamma_i(H)  \\  \lmapdown{  } &&\rmapdown{ }
 \\   \gamma_i(R )  &  \mapright{ } &\gamma_i(S ) }$$  where the rightward arrows are induced by $\phi$ and $\mu$ respectively  and downward arrows   are epimorphisms, therefore $\gamma_i(S )^{n_i} \le \mu (\gamma_i(R ))$ and since $\tau (R)\subseteq \tau,$ $n_i$ is a $\tau (R)'$-number.  Hence   $S$ is para-$R$ via $\mu,$ by Theorem 2.2.

\vq Suppose that $\alpha : G\rightarrow H$ is an isomorphism. Then $$\alpha (\gamma_k (G))=\gamma_k(H) \Longrightarrow \alpha(M)=N\Longrightarrow \alpha (Z_j (M))=Z_j (N)$$ and $\alpha$ induces an isomorphism $G/Z_j (M)\cong H/Z_j (N).$     $\square$
\end{proof}
 
\end{proof}
		\section{The nilpotent genus of polycyclic groups }  
	  
\vq Every polycyclic group has a characteristic subgroup of finite index which is both residually nilpotent and nilpotent by abelian so for the rest of this paper we will focus on   the class ${\cal P}^*$ of polycyclic groups which are residually nilpotent and nilpotent by abelian. We will denote by ${\cal P}_c^*$ the class of polycyclic groups which are residually nilpotent and nilpotent of class $c$ by abelian. 
	
 \begin{lemma} If $  G, H $ are residually nilpotent polycyclic groups   and $H$ is para-$G$ via $\phi$, then $$G\in {\cal P}_c^* \Longleftrightarrow  H\in {\cal P}_c^*.$$ 
		\end{lemma}
		\begin{proof}  If  $G\in {\cal P}_c^* $, then $G'$ is nilpotent of class $c$. If   $h_1,\dots ,h_{c+1}\in H',$ then given any $i$, there exists $g(i)_j\in G'$ such that $h_j\gamma_{i+1} (H)=\phi(g(i)_j )\gamma_{i+1}  (H),$ for all $1\le j\le c+1,$  and  $$\ba{ll} & [h_1,\dots ,h_{c+1}] \gamma_{i+1} (H) = \phi([g(i)_1,\dots,g(i)_{c+1}])\gamma_{i+1} (H) =1\gamma_{i+1} (H) \\ & \\ \Longrightarrow & [h_1,\dots ,h_{c+1}]\in \ds  \bigcap_{i=1}^{\infty} \gamma_{i} (H) =1.\ea$$   Therefore $H'$ is nilpotent of class $c$ and  $ H\in {\cal P}_c^*.$  Conversely  suppose that $H\in {\cal P}_c^*,$ then  $G\cong \phi (G) \in {\cal P}_k^*,$ for some $k\le c.$  This would imply that $H\in {\cal P}_k^*$ and therefore $k=c.$   $\square $
\end{proof}
		 \vskip 0.5cm Baumslag, Mikhailov and Orr studied the nilpotent genus of groups  $G,H\in {\cal P}_1^*$, and showed that if $H$ is para-$G$, then $G$ is para-$H$ so that this is an equivalence relation. They did this by constructing the telescope of a metabelian group which is a type of group localization. We will give the bare details here but further details can be found in \cite{bmo2}. Let $G\in {\cal P}_1^*$ be a metabelian group with derived subgroup $A,\, Q=G_{ab}, $ integral group ring $R=\Z Q,$   augmentation ideal $I$ and set  $S=1+I.$ The ring $R_S$ is the set of equivalence class of elements $(a,s)\in R\times S$ subject to the equivalence relation $$(a,s)\sim (b,t) \; \;{\rm if \; there \; exists\;} u\in S\;\;{\rm such \; that}\; (at-bs)u=0.$$ The equivalence class of $(a,s)$ is denoted by $a/s.$ Likewise $A_S$ is the set of equivalence class of elements $(a,s)\in A\times S$ subject to the equivalence relation $$(a,s)\sim (b,t) \; \;{\rm if \; there \; exists\;} u\in S\;\;{\rm such \; that}\; (at-bs)u=0.$$ The equivalence class of $(a,s)$ is denoted by $a/s$ and $A_S$ is turned into an  $R_S$-module in the obvious way by defining $a/s \cdot r/t =ar/st .$ If we fix $s$ and consider the $R$-module $A_s=\{a/s\mid a\in A\}$, then the mapping $\tau_s:a\mapsto a/s$  maps $A$ isomorphically onto $A_s\le A_S$ provided that $as\not=0$, for every $a\in A\backslash \{0\}.$  Since $G$ is residually nilpotent, $S$ does not contain any zero divisors of $A$ so that  $\tau_s$ is monic and also the mapping $\tau:A\rightarrow A_S$ is monic where $a\mapsto a/1.$   If $\alpha\in \Hom_R(A, B)$ where $B$ is also an $R$-module, then we can extend it to $\alpha_S\in \Hom_{R_S}(A_S,B_S).$  We can form the semi-direct product $P=A_S\rtimes G$ and setting $K=\{(a/1,a^{-1})\mid a\in A\}$, the telescope of $G$ is the factor group $G_S= P/K.$  Viewing $G$ and $A_s$ as subgroups of $G_S$, set $G_s= GA_s.$  If $S=\{s_1,s_2,\dots \}$,   $t_n=s_1s_2\dots s_n$, $G_0=G$ and  $G_i=G_{t_i}$, then,   by Lemma 5.6 \cite{bmo2}, $G\cong G_i$, for all $i$ and $G_S=\bigcup_{i=1}^{\infty} G_i .$  
		  
\begin{theorem}  Let $G,H\in {\cal P}_1^*$ and    $\phi:G\rightarrow H$   a homomorphism that induces an isomorphism  $$ \phi_2:G_{ab}\rightarrow H_{ab}. $$   Then $h(H)\le h(G).$ 
\end{theorem} 
\begin{proof} By Lemma 2.1  (i) $\phi$ induces epimorphisms $\phi_i :G/\gamma_i(G)\rightarrow H/\gamma_i (H).$   As in Theorem 5.8 \cite{bmo2},  $\phi$ induces an epimorphism $\phi_S:G_S\rightarrow H_S.$  By Lemma 5.6 \cite{bmo2},  $G_S$ is a union of subgroups $G_i$ with    $G\cong G_i.$ Since  $H$ is finitely generated, we can find a $k$ such that $H\le \phi_S(G_k)$ and $h(H) \le h(G_k) =h(G).$ $\square$

\end{proof}

\begin{cor} Let $G,H\in {\cal P}^*$ and    $\phi:G\rightarrow H$   a monomorphism that induces an isomorphism   $\phi_2:G_{ab}\rightarrow H_{ab} .$ Suppose that $N\triangleleft H$ where $N\le H'$ and  $\tilde H=H/N\in {\cal P}_1^*.$ Set  $M=\phi^{-1} (N) $  and $\tilde G= G/M,$  so that  $\phi$ induces a monomorphism $\mu:\tilde G \rightarrow  \tilde H. $  Then $\tilde G  \in {\cal P}_1^*,$   $h(\tilde G)=h(\tilde H).$   
\end{cor} 
\begin{proof} It is straightforward to see that $\mu$ is well-defined and is a monomorphism and hence $\tilde G  \in {\cal P}_1^*$.  By the hypothesis $N\le H'$ and $\phi_2 :G_{ab}\cong H_{ab}$ so that  $M\le G'$, $$\tilde G_{ab} \cong G_{ab} \cong H_{ab}\cong \tilde H_{ab}$$ and $\mu$ induces an isomorphism $\mu_2:\tilde{G}_{ab} \rightarrow \tilde{H}_{ab} .$   Both $\tilde{G}$ and $\tilde{H}$ are metabelian and residually nilpotent so that $\mu$ satisfies the hypothesis of Theorem 3.2 and hence $h(\tilde{H})\le h(\tilde{G}).$ Since $\mu$ is a monomorphism  this implies that $h(\tilde{H})= h(\tilde{G}).$ 
   $\square$
	\end{proof}

 \begin{theorem}  Let $G,H\in {\cal P}_c^*$ where $H$ is para-$G$ via $\phi$ and one of the following holds: \bi \item[(i)] $    G/G'',\;  H/H''$ are residually nilpotent;
\item[(ii)]  $|Z_{c-1} (G'): G''|,\; | Z_{c-1} (H'): H''|<\infty .$ \ei Then $h(H)=h(G)$.
\end{theorem}
\begin{proof} We will first show in  both cases that $h(G /G'')=h(H /H'').$
 \vq (i)  It is straightforward to see  that $\phi(G'')\le H''.$ Suppose that  $\phi (x) \in H'' ,$  then  $$\phi(x) \gamma_i (H) \in H''\gamma_i (H)/\gamma_i (H) = (H/\gamma_i(H))'' ,$$ for all $i$,  and, as $\phi_i $ is an isomorphism, this implies that $$x\gamma_i (G)\in G''\gamma_i(G)/\gamma_i(G)\Longrightarrow x\in \bigcap_{i=1}^\infty   G''\gamma_i (G)=G''.$$ Therefore $\phi^{-1} (H'')=G''$ and,   
by Corollary  3.3,   $h(G/G'')=h(H/H'').$     
\vq (ii) By Lemma 2.4, $H/Z_{c-1}(H') \in {\cal P}_1^*,$  by Lemma 2.5, $\phi^{-1} (Z_{c-1} (H'))=Z_{c-1} (G')$ and so,  by  Corollary 3.3, $h(G/Z_{c-1} (G'))=h(H/Z_{c-1} (H')).$       By the hypothesis   $$h(G/G'')=h(G/Z_{c-1} (G'))=  h(H/Z_{c-1} (H'))=h(H/H'').$$  
 \vskip 0.25 cm In both cases  $h(\phi(G)H''/H'')=h(G/G'')=h(H/H''). $  Since $\phi_2:G_{ab}\rightarrow H_{ab}$ is an isomorphism,    $$h(G_{ab})=h(H_{ab})\Longrightarrow h(G'/G'')=h(G/G'')-h(G_{ab})=h(H/H'') -h(H_{ab})=h(H'/H'').$$    Therefore $\phi(G')\gamma_2(H')/\gamma_2(H')$ has finite index in $H' /\gamma_2(H')$. As $H'$ is a finitely generated nilpotent group,   $\phi (G')$ has finite index in $H',$ by Lemma 2.1 (ii), so that $h(G')=h(H')$ and   $$h(G)=h(G')+h(G_{ab})=h(H')+h(H_{ab})=h(H).$$
 $\square$
\end{proof}
\vh When $h(G)=h(H)$ we have  the following  partial converse to Theorem 2.2: 
\begin{lemma}  Let $G,H\in {\cal P}_c^*$ where  $H$ is para-$G$ via $\phi$ and  $h(G)=h(H)$. Then $\phi$ is a lower central $\pi$-monomorphism, for some set of primes $\pi$ such that $\pi'$ is a finite set.

\end{lemma}
\begin{proof}  By  definition $\phi$ induces an isomorphism $G/\gamma_i(G)\cong H/\gamma_i (H)$, for any $i\ge 2$,  and so  $h(\gamma_i (G))=h(\gamma_i(H)).$ Since $\gamma_i(H)$ and $\gamma_i(G)$ are finitely generated nilpotent groups, this implies that 
$\gamma_i(H)^{n_i}\le \phi(\gamma_i (G))$, for some $n_i$.  By Lemma 2.1 (iv),   there exists $k$ such that $|\gamma_k(G)/\gamma_{k+j} (G)|<\infty,$ for all $j\ge 1$, and   $\pi(m_j)\subseteq \pi(m_1)$ where $m_j= \exp(\gamma_k(G)/\gamma_{k+j} (G))$.  If $h\in \gamma_{k+j} (H),$ then $h^{n_km_j}\in \phi(\gamma_{k+j} (G)).$ Set $\rho =\displaystyle \bigcup_{i=1}^k \pi(n_i)\cup\pi(m_1)  $ and $\pi=\rho',$  then $\phi$ is a lower central $\pi$-monomorphism.   $\square$
\end{proof}
\vq Note we cannot stipulate that $\tau (G)\subseteq \pi$. Let  $G_1,\,H_1\in {\cal P}_c^*$ where    $h(G_1)=h(H_1),$ $\tau_1 =\tau (G_1)$ and $\mu:G_1\rightarrow H_1$ is  a lower central $\tau_1$-monomorphism. Then $H_1$ is para-$G_1$ via $\mu$, by Theorem 2.2. Suppose that $p|n_2$ where $\gamma_2 (H_1)^{n_2}\le \mu (\gamma_2(G_1)) $ and $n_2$ is a $\tau_1'$-number. Set $G=C_p\times G_1,\,H=C_p\times H_1$ and define $\phi:G\rightarrow H$ by $$(x,y)\mapsto (x,\mu(y)).$$ Then $H$ is para-$G$ via $\phi,$ $h(G)=h(H)$ and $\gamma_2(H)^{n_2} \le \phi( \gamma_2(G))$  but $\tau=\tau(G)=\tau_1\cup \{p\}$ so that $\phi$ is not a  lower central $\tau$-monomorphism.

\vh    In   \cite{bmo2},  Baumslag, Mikhailov and Orr  found  a finitely generated abelian by cyclic  group $G$ such that the nilpotent genus of $G$ is nontrivial. The following   results  have been proven with an aim to extend their example. 

\begin{lemma} Let $N$ be a finitely generated torsion-free nilpotent group of class $c$ and suppose that    $\;\;\;\;\;$    $<\alpha_i\mid 1\le i\le r> \le \Aut (N)$ is  abelian   and  there exists  $a_i\in \N$  such that  $$\ds\Pi_{j=0}^{a_i-1} \alpha_i^{j} (x)\in N',$$ for all $x\in N,\; 1\le i\le r.$  Let $A$ be the free abelian group  on $\{t_1,\dots , t_r\}$ and set $x^{t_i}=\alpha_i (x)$, for all $x \in N,  $  and   $H= N\rtimes A $ using this action.  If $a=\gcd (a_i)$ and $b={\rm lcm} (a_i),$ then  $H$ is a residually finite $b$-group, $N=I_{\pi (a)} (H') $  and $\tau (H)=\pi(a).$  In particular if $a=b=p$ is a prime, then $H\in{\cal P}^*$.\end{lemma}

\begin{proof}  If  $x\in N$,  then  $$1N'=x^{t_i^{a_i-1}+\dots +t_i+1}N' =x [x,t_i^{a_i-1}]\dots x[x,t_i] xN' \Longrightarrow\ds x^{a_i} \in   H'  \Longrightarrow x^a \in H' $$  and $N=I_{\pi(a)} (H').$  If $x,y\in N,\;u,\,v\in A,$ then  $$\ba{ll} \ds [xu, yv]^a\gamma_3(H) &= ([x,v][x,y] [u ,y])^a \gamma_3 (H)= [x,v]^a[x,y]^a [u ,y]^a \gamma_3 (H) \\ &\\ &\ds =[x^a,v] [x^a,y][u,y^a]\gamma_3 (H)=1\gamma_3(H) \ea$$ so that $\gamma_{2}(H)/\gamma_{3}(H) $ is a finite group of exponent dividing $a$ and $\tau (H)=\pi(a).$  

\vskip 0.25cm Given any $x\in N, \,1\le i\le r,$ $$x^{t_i^{a_i}-1}=x^{(t_i^{a_i-1}+\dots +t_i+1)(t_i-1)}\in N'\Longrightarrow  x^{t^b}N'=xN',$$ for all $t\in A,$     and   $K=N\rtimes A^b$ is a finitely generated torsion-free nilpotent group of class $c$, by Lemma 2.1 (iii) and hence is a residually finite $b$-group. Since $K\triangleleft H$ and   $H/K$ is a finite $b$-group,  $H$ is also a residually finite $b$-group.  If $b=p,$ then the finite $p$-quotients are nilpotent so that $H$ is a residually nilpotent group. It is straightforward to see that $H$ is polycyclic and nilpotent by abelian  so  $H\in{\cal P}^*.$  $\square$ \end{proof} 

 From now on we will be considering the case where   $N$ is a finitely generated torsion-free nilpotent group and $\alpha\in \Aut (N).$ We will denote by  $T$  the infinite cyclic group on $t$  and set  $H=N\rtimes T$ where $t$ acts on $N$ by $\alpha$. We also consider  a subgroup  $M\le N$ where $M\triangleleft H$ and we set $G=M\rtimes T.$  We will also assume that $H$ is not nilpotent. 
\begin{theorem} Let  $N$ be  a finitely generated torsion-free nilpotent group of class  at most  $2$,    $\alpha\in \Aut (N)   $  such that   $$\ds\Pi_{j=0}^{p-1} \alpha^{j} (y)\in N',$$ for all $y\in N$ and some prime $p,$ and  $N_{ab}=<\overline\alpha^i (a) \mid 0\le i\le p-2>,$ for some $a\in N_{ab},$  where  $\overline\alpha$ is  the induced automorphism  in $ \Aut (N_{ab}).$         If $ M\triangleleft  H, $ $\ds\Pi_{j=0}^{p-1} \alpha^{j} (y)\in M',$ for all $y\in M$ and   $N^n\le M\le N$, for some $n$ prime to $p,$   then $H$ is para-$G$ via  inclusion.
\end{theorem} 
\begin{proof} Let  $x\in N$ where  $xN'=a$. Then,  by Lemma 2.1 (i),  $$ N=<x^{t^i} \mid 0\le i\le p-2>.$$    By Lemma 3.6, $G$ and $ H$ are residually nilpotent groups, $\tau (G)=\tau(H)=p,$ $ M=I_p (G'),$ $ N=I_p (H')$  and  $H_{ab} \cong C_p \times <t>.$  In order to apply Corollary 2.3, we   will show that the inclusion of $G$ into $H$ induces an isomorphism $G_{ab}\cong H_{ab}$ and  $\gamma_2 (H)^{n} \le \gamma_2 (G).$     Since $M^p \le G'$ and, by Lemma 2.1 (ii),  $(N')^m\le  M'\le G',$ for some $m$ prime to $p,$ it follows  that $M\cap N'\le G'.$   We can find $\alpha $ such that $\alpha n \equiv 1\,\mod \,p.$ If  $y \in M, $ then $y = x^{k_1t^{l_1}} \dots x^{k_rt^{l_r}},$ for some $r, k_j,\,l_j \in \Z$ and    $$\ba{ll} &\ds y^{\alpha n}  x^{-\alpha nk_1t^{l_1}} \dots x^{-\alpha nk_rt^{l_r}} \in M\cap N' \le G'\\ &\\ \Longrightarrow & \ds yG'=y^{\alpha n}G' =  x^{ \alpha nk_1t^{l_1}} \dots x^{ \alpha nk_rt^{l_r}} G'= x^{\alpha nk}G'\ea $$ where $\ds k=\sum_{j=1}^r k_j.$ Therefore    $G_{ab} \cong C_p \times <t>$ and the inclusion of $G$ in $H$ induces an isomorphism $G_{ab}\cong H_{ab}$ where  $x^{an} t^b G' \mapsto  x^{an}t^b H'.$ If $h\in H'\le N,$ then $h^n \in M$ and, since  $G_{ab}\cong H_{ab},$ $h^n\in G'.$ 
 Therefore, by Corollary  2.3, $H$ is para-$G$ via the inclusion.    $\square$

\end{proof} 
 \begin{theorem} Let  $N$ be  a  finitely generated free nilpotent group of class $2.$   Suppose that  $\mu\in \Aut (N_{ab})   $       has order $p,$ for some prime $p,$   is fixed-point free and  $N_{ab}=<\mu^i (a) \mid 0\le i\le p-2>$, for some $a\in N_{ab},$ then $\mu$ lifts to an automorphism  $\alpha\in Aut (N).$   Let   $  \tilde H $ be the semi-direct product  of $  N_{ab}$ by $T$ with $t$ acting on $ \, N_{ab}$ by $ \mu$.   Suppose that $A\le N_{ab}$      where $A\triangleleft \tilde H$ and $N_{ab}^m\le A$, for some $m$ prime to $p,$ set $\tilde G=A\rtimes T.$ If $M\le N$ where $M\triangleleft H, \;M\cap N'=M'$ and   $MN'/N' =A$,  then $H$ is para-$G$ where $G=M\rtimes T.$  Furthermore if $\tilde G\not\cong \tilde H,$ then $G\not\cong  H.$ 
\end{theorem}  
\begin{proof}  By  Theorem 2.1 \cite{and}, $\mu$ can be lifted to $\alpha \in \Aut (N).$   Since $\mu$ has  order $p$ and is fixed-point free,
$\ds\Pi_{j=0}^{p-1} \mu^{j} (b)=1,$ for all $b\in N_{ab}$ and therefore  $\ds\Pi_{j=0}^{p-1} \alpha^{j} (g)\in N',$ for all $g\in N$.  If $g \in M, $ then  $\ds\Pi_{j=0}^{p-1} \alpha^{j} (g)\in M\cap   N' =M'.$  Both $M$ and $N$ satisfy the hypothesis of Lemma 3.6, therefore  $M=I_p (G'),\, N=I_p(H').$ From the hypothesis, $N_{ab}^m\le A$ which implies  that $N^m \le MN' $ and  therefore $N^n\le M$,  for some $\pi (m)$-number $n$, by Lemma 2.1 (ii).   Hence    $\tilde H$ is para-$\tilde G$ and  $H$ is para-$G$, by Theorem 3.7.   Since $M'=M\cap N'=M\cap Z(N),$ by Lemma 2.5, $M'=Z(M)$ and   $$\tilde G= A\rtimes T = MN'/N' \rtimes T =  M/M'\rtimes  T  =M/Z(M)\rtimes T=G/Z(M).$$   Hence, by Proposition 2.6,  
  if   $\tilde G\not\cong \tilde H,$  then  $G\not\cong H.$  $\square$
\end{proof}

\section {Pronilpotent Completions} 
\par The aim in this section is to extend the Baumslag, Mikhailov and Orr result, Theorem 8.1 \cite{bmo2}, that the pro-nilpotent completion of  a residually nilpotent  metabelian polycyclic group is locally polycyclic  to all   groups in ${\cal P}^*.$ Given any group $G$,     the pro-nilpotent completion of $G$, $\hat G_{nil},$ is the inverse limit of the lower central quotient groups of $G:$ $$\hat G_{nil}= \lim_{\leftarrow } G/\gamma_i (G).$$ Some obvious facts are:
\bi\item[(i)] if $G$ is a residually nilpotent group, then $G$ embeds in $\hat G_{nil}$; \item[(ii)] if $G$ is soluble of derived length $d$, then $\hat G_{nil}$ is also soluble of derived length $d$; \item[(iii)] if $H$ is para-$G$, then $\hat G_{nil}\cong \hat H_{nil}.$ \ei
Recall that a group $G$ is said to be HNN-free if it has no subgroups that are nontrivial HNN-extensions
\begin{lemma} Let $G$ be a finitely generated soluble group then the following properties are equivalent:
\bi \item[(a)] $G$ is polycyclic; \item[(b)]   $<x>^{<g>}$ is finitely generated for all $x , \,g \in G$;   \item[(c)] $G$ is HNN-free.\ei
\end{lemma}
\begin{proof} (a) $\Leftrightarrow$ (b) by 4.6.4 \cite{lr}.
\vq (b) $\Leftrightarrow$ (c) Lemma 1 \cite{rrv}  $\square$
\end{proof}

\begin{lemma} Let $G$ be a group, let $N\triangleleft G$ and set
$Q=G/N.$ If $Q$ is HNN-free, then the following properties are equivalent:
\bi \item[(a)] $G$ is HNN-free;
\item[(b)] if $g\in G$ and $H^g\le H\le N,$ then $H=H^g;$
\item[(c)] $<x>^{<g>}$ is finitely generated for all $x\in N,\,g \in G.$
\ei

\end{lemma}
\begin{proof} (a)$\Rightarrow $ (b) and (c), by Lemma 1 \cite{rrv}.

\vq (b) $\Rightarrow$ (a) By Lemma 1, \cite{rrv} $G$ is HNN-free if, and only if,   $K^g\le K\le G$ implies that $K^g=K.$  
Since $Q$ is HNN-free, $K^g=K$ if, and only if, $(K\cap N)^g=K\cap N .$
\vq (c) $\Rightarrow$ (b) Suppose $H^g\le H\le N,$ for some $g \in G.$ If $x \in H,$ then  $<x>^{<g>}$ is finitely generated and hence there exists $n>0$ such that  $$<x>^{<g>}=\left< x,x^{g^{-1}}, \dots ,x^{g^{-n}}\right>.$$ Therefore $x^{g^{-n-1}}\in H^{g^{-n}}$ and $x \in H^g.$ $\square$ 
\end{proof}

\begin{lemma} Let $G$ be  polycyclic, $A \triangleleft G$ be abelian and   $Q=G/A,$ we view as $A$ as a $\Z Q$-module. Then for each $t=xA\in Q$, there exists polynomials $\alpha,  \beta$ $$\ba{l}  \alpha (t)=c_0 +c_1 t+ \dots+c_{m-1}t^{m-1}\pm t^m,\\ \\ \beta(t)=  d_0 +d_1 t+ \dots+d_{n-1}t^{n-1} \pm t^n \ea$$
such that $$a\alpha(t)=a \beta (t^{-1})  =0,\;\;{\rm for \; all \;} a\in A.$$ Furthermore $<a>^{<x>}$ is finitely generated by  $m+n-1$ conjugates of $a$ by powers of $x.$ 
 \end{lemma}
\begin{proof} As in Lemma 8.3 \cite{bmo2}, given any $a\in A$ we can find $\alpha_a,\,\beta_a,$ such that $a\alpha_a (t)=a \beta_a (t^{-1})  =0.$
 Therefore   $$a(\alpha_a(t)t^j)=(a\alpha_a(t))t^j=0 =(a  \beta_a (t^{-1}))t^{-j} =a(\beta_a(t^{-1})t^{-j}) , \;\; \forall j \ge 0.$$ 
 As $G$ is polycyclic, $A$ is finitely generated as an abelian group, say  $A=<a_1,\dots ,a_n>.$ Let $$\alpha=\alpha_{a_1}\dots \alpha_{a_n},\;\beta= \beta_{a_1}\dots \beta_{a_n}.$$ Then $a_i\alpha(t)=a_i \beta(t^{-1}) =0$, for all $1\le i\le n$, and hence $a\alpha(t)=a \beta(t^{-1})  =0,\;\;{\rm for \; all \;} a\in A.$ Note  we can rewrite $a\alpha (t)=0=a \beta( t^{-1}) $ using group notation as $$\ba{ll} \ds 1=a^{\alpha(x)}=a^{c_0 +c_1 x+ \dots+c_{m-1}x^{m-1}\pm x^m}=a^{c_0}a^{c_1 x} \dots a^{c_{m-1}x^{m-1}}a^{\pm x^m} \\ \\ \ds 1=a^{\beta(x^{-1})}=a^{d_0 +d_1 x^{-1}+ \dots+d_{n-1}x^{-(n-1)}\pm x^{-n}}=a^{d_0}a^{d_1 x} \dots a^{d_{n-1}x^{n-1}}a^{\pm x^{-n}} \ea$$  and so  $$<a>^{<x>}=<a^{x^{-(n-1)}},\dots ,a^{x^{-1}}, a,a^x,\dots,a^{x^{m-1}}> .$$$\square$ \end{proof}

\begin{theorem} If     $G\in {\cal P}^*$, then $\hat G_{nil}$ is locally polycyclic.

\end{theorem}
\begin{proof} Let    $A=Z(G')$ and  $Q=G/A.$  Then, by Lemma 2.4 , $Q$ is residually nilpotent and hence $Q\in {\cal P}^*$. Note if $h(A)=0, $ then $A$ and hence $G'$ are finite. As $G$ is residually nilpotent, this would in fact imply that $G$ is nilpotent. The result  would  follow  trivially  so we shall assume that $h(A)\not= 0$ and hence $h(Q)<h(G).$    Set  $$ \ba{l} A_i =A \gamma_i (G)/ \gamma_i (G),\; G_i=G/\gamma_i(G),\;Q=Q/\gamma_i(Q), \\ \ds  \tilde A= \lim_{\leftarrow } A_i,\; \hat G_{nil}= \lim_{\leftarrow } G_i,\; 
 \hat Q_{nil}= \lim_{\leftarrow } Q_i,\ea $$   then $\tilde A$ embeds into $\hat G_{nil},$ we will now view it as a subgroup, and $\hat G_{nil}/\tilde A$ embeds into $\hat Q_{nil}$. We will use induction on the Hirsch length, $h(G)$, of $G$. By the remarks above we are assuming that $h(Q)<h(G).$ If $h(G)=1,$ then $Q$ is finite and so is nilpotent. Therefore $\hat Q_{nil} =Q$ and   $\hat G_{nil}$ is  an abelian-by-finite group and hence locally polycyclic.   Suppose $h(G)>1,$  as  $Q\in {\cal P}^*$,  $Q$ satisfies the hypothesis and 
 $h(Q)<h(G),$  we can assume that $\hat Q_{nil}$ is locally polycyclic. By Lemma 4.1 and Lemma 4.2, $\hat G_{nil}$ is locally polycyclic, if, and only if, $<a>^{<x>}$ is finitely generated, for all $a \in \tilde A,\; x\in \hat G_{nil}.$ We follow the notation of Lemma 8.5, \cite{bmo2}. Given  $a \in \tilde A,\; x\in \hat G_{nil},$ let $$ a(n)=a_{1}a_{2}\dots a_n  \gamma_{n+1}(G ),\;\;x(n)=x_1x_2\dots x_n\gamma_{n+1} (G) $$  where $a_j \in A\cap \gamma_j (G) $ and  $x_j\in \gamma_j (G).$   Then $$(a^x)(n)=a(n)^{x(n)}=(a_1\dots a_n)^{x_1\dots x_{n}}\gamma_{n+1} (G) =(a_1\dots a_n)^{x_1}\gamma_{n+1} (G) .$$    Set $t=x_1A\in Q, $ then, by Lemma 4.3, there exists $\alpha$ and $\beta$ such that  $$b\alpha(t)=b \beta(t^{-1}) =0,\;\;{\rm for \; all \;} b\in A.$$  Therefore $$\ba{l  }  \ds  (a^{\alpha(x)})(n)= a(n)^{\alpha(x(n))}=(a_1\dots a_n)^{\alpha(x_1)}\gamma_{n+1} (G) =1\gamma_{n+1}(G) ,\\  \\(a^{\beta(x^{-1})})(n)=a(n)^{ \beta(x(n)^{-1} )} =(a_1\dots a_n)^{ \beta(x_1^{-1})}\gamma_{n+1} (G)  =1\gamma_{n+1}(G) ,\;\;\forall n\in \N\\  \\ \Longrightarrow   a^{\alpha(x)}=a^{ \beta (x^{-1}) } =1.\ea$$
 Hence $<a>^{<x>}$ is finitely generated and    $\hat G_{nil}$ is locally polycyclic.  $\square$

    \end{proof}

\printbibliography 

\vq \textsc{Niamh O'Sullivan, School of Mathematical Sciences, Dublin City University,  Dublin 9, Ireland.}
\vq\textit{Email Address:} \texttt{niamh.osullivan@dcu.ie}		
\end{document}